\documentclass[11pt,a4paper]{article}

\usepackage[hyperfootnotes=false,pdfpage layout=SinglePage]{hyperref}
\usepackage[top=2.9cm,bottom=4cm,left=3.6cm,right=3.6cm]{geometry}
\usepackage{setspace}
\usepackage{amssymb}
\usepackage{amsmath}

\usepackage{amsthm}
\usepackage{latexsym}
\usepackage{amsfonts}
\usepackage{mathrsfs}

\newtheoremstyle{my_theoremstyle}
  {5mm}
  {1mm}
  {\it}
  {0pt}
  {\bfseries}
  {{\bf .} }
  {0mm}
  {}
\newtheoremstyle{my_definitionstyle}
  {5mm}
  {1mm}
  {\rm}
  {0pt}
  {\bfseries}
  {{\bf .} }
  {0mm}
  {}
\theoremstyle{my_theoremstyle}
\newtheorem{thm}{Theorem}[section]
\newtheorem{cor}[thm]{Corollary}
\newtheorem{lem}[thm]{Lemma}
\newtheorem{prop}[thm]{Proposition}
\theoremstyle{my_definitionstyle}
\newtheorem{defn}[thm]{Definition}
\newtheorem{rem}[thm]{Remark}
\newtheorem{ex}[thm]{Example}

\renewenvironment{proof}{\noindent{\bf Proof. }}{\vspace{.2cm}\qed}

\renewenvironment{enumerate}
{
\begin{list}{}
{\setlength{\topsep}{-.3cm} \setlength{\parsep}{0cm} \setlength{\itemsep}{.1cm} \setlength{\leftmargin}{1.0cm} \setlength{\labelwidth}{1.0cm}}
}
{\end{list}}

\numberwithin{equation}{section}
\makeatletter
\renewcommand\section{\@startsection{section}{1}{\z@}%
                                  {6ex \@plus 1ex \@minus 1ex}%
                                  {0.2ex \@plus.2ex}%
                                  {\normalfont\large\bfseries}}
\renewcommand\subsection{\@startsection{subsection}{1}{\z@}%
                                  {-5ex \@plus -1ex \@minus -.2ex}%
                                  {0.1ex \@plus.2ex}%
                                  {\normalfont\large\it}}
\makeatother
\newcommand{\paragraf}{\textsection}
\renewcommand{\emptyset}{\varnothing}

\newcommand{\R}{\ensuremath{\mathbb R}}    
\newcommand{\C}{\ensuremath{\mathbb C}}    
\newcommand{\N}{\ensuremath{\mathbb N}}    

\newcommand{\gperp}{{[\perp]}}

\newcommand{\product}{[\cdot\,,\cdot]}
\newcommand{\hproduct}{(\cdot\,,\cdot)}

\newcommand{\calH}{\mathcal H}

\newcommand{\calK}{\mathcal K}         
\newcommand{\calL}{\mathcal L}         
\newcommand{\calM}{\mathcal M}

\newcommand{\calU}{\mathcal U}

\newcommand{\la}{\lambda}
\newcommand{\veps}{\varepsilon}

\newcommand{\mat}[4]
{
   \begin{pmatrix}
      #1 & #2\\
      #3 & #4
   \end{pmatrix}
}

\renewcommand{\Im}{\operatorname{Im}}

\renewcommand{\ker}{\operatorname{ker}}
\newcommand{\ran}{\operatorname{ran}}
\newcommand{\dom}{\operatorname{dom}}

\newcommand{\sap}{\sigma_{{ap}}}

\renewcommand{\sp}{\sigma_{+}}
\newcommand{\sm}{\sigma_{-}}

\newcommand{\lra}{\longrightarrow}

\newcommand{\Lra}{\Longrightarrow}

\newcommand{\Llra}{\Longleftrightarrow}
\newcommand{\Slra}{\Leftrightarrow}

\newcommand{\downto}{\downarrow}

\newcommand{\ol}{\overline}
\newcommand{\ds}{\dotplus}
\newcommand{\wt}{\widetilde}

\begin{document}
\thispagestyle{empty}
\vspace*{-.3cm}
\begin{center}
\begin{spacing}{1.5}
{\LARGE\bf Spectral functions of products of selfadjoint operators}
\end{spacing}

\vspace{1cm}
{\Large Tomas Ya.\ Azizov, Mikhail Denisov, Friedrich Philipp}

\end{center}

\vspace{.7cm}
{\bf Abstract:} Given two possibly unbounded selfadjoint operators $A$ and $G$ such that the resolvent sets of $AG$ and $GA$ are non-empty, it is shown that the operator $AG$ has a spectral function on $\R$ with singularities if there exists a polynomial $p\neq 0$ such that the symmetric operator $Gp(AG)$ is non-negative. This result generalizes a well-known theorem for definitizable operators in Krein spaces.

\vspace{.3cm}
{\it Keywords:} Product of selfadjoint operators, indefinite inner product, definitizable operator



\setlength{\parskip}{3ex plus 0.5ex minus 0.2ex}

\section{Introduction}
Let $A$ and $G$ be two selfadjoint operators in a Hilbert space
$(\calH,\hproduct)$ such that either $A$ or $G$ is bounded and
boundedly invertible. Then the product $AG$ is selfadjoint in a
Krein space. Indeed, if $G$ ($A$) is bounded and boundedly
invertible, then $AG$ is selfadjoint in the Krein space
$(\calH,\product_G)$ ($(\calH,\product_{A^{-1}})$, respectively),
where
$$
[x,y]_G = (Gx,y),\quad [x,y]_{A^{-1}} = (A^{-1}x,y),\quad x,y\in\calH.
$$
Conversely, a selfadjoint operator in a Krein space can be written as a product of two selfadjoint operators in a Hilbert space one of which is bounded and boundedly invertible.

The spectrum of a selfadjoint operator in a Krein space is symmetric with respect to the real axis. But even simple examples show that the spectrum of such operators can be empty or cover the entire complex plane. However, some classes of selfadjoint operators in Krein spaces are well-understood. Among those are the definitizable operators. A selfadjoint operator $T$ in the Krein space $(\calH,\product)$ is called definitizable if its resolvent set $\rho(T)$ is non-empty and if there exists a polynomial $p\neq 0$ with real coefficients such that
$$
[p(T)x,x]\,\ge\,0\quad\text{for all }x\in\dom p(T).
$$
This definition goes back to H.\ Langer who proved that the spectrum
of a definitizable operator $T$ -- with the possible exception of a
finite number of non-real eigenvalues which are poles of the
resolvent of $T$ -- is real and that $T$ possesses a spectral
function on $\R$ with a finite number of singularities, see
\cite{l}. Definitizable operators appear in many applications including
differential operators with indefinite weights (see, e.g.,
\cite{abt,bp,cl,km,kt,kwz}), selfadjoint operator polynomials (see,
e.g., \cite{dl,l_h}) and Sturm-Liouville equations with floating
singularity (see, e.g., \cite{jt02,jt06,lmem}).

In the present paper we extend the spectral theory of definitizable
operators from selfadjoint operators in Krein spaces to products $T
= AG$ of selfadjoint operators $A$ and $G$ in a Hilbert space which
are both allowed to be unbounded and non-invertible. Instead of
$\rho(T)\neq\emptyset$ as in the above definition of
definitizability we will have to assume that both resolvent sets
$\rho(AG)$ and $\rho(GA)$ are non-empty. In the Krein space case
(i.e.\ when $G$ or $A$ is bounded and boundedly invertible) this is
equivalent to $\rho(T)\neq\emptyset$ since in this situation the
operators $AG$ and $GA$ are similar. As is shown in Theorem
\ref{t:main}, a definitizable product $AG$ of selfadjoint operators
$A$ and $G$ has the same above-mentioned spectral properties as a
definitizable operator in a Krein space. Moreover, it has a spectral
function with a finite number of singularities, see Theorem
\ref{t:sf}. In the special case when both $A$ and $G$ are bounded,  
$A\geq{0}$ and $0\notin{\sigma_{p}(A)}$ the existence of a spectral 
function of $AG$ was already proved in \cite{denm}.

The techniques used in our proof of the existence of the spectral
function are different to those in \cite{l} where an analogue of
Stone's formula for selfadjoint operators in Hilbert spaces was used
to define the spectral function. Here, we make use of the concept of
the spectral points of positive and negative type of symmetric
operators in inner product spaces which was introduced by H.\
Langer, A.S.\ Markus and V.I.\ Matsaev in \cite{lmm}, see also
\cite{ajt,abjt,lamm} for the Krein space case. In Theorem
\ref{t:main} we prove that if the product $AG$ is definitizable,
then there exists a finite number of real points which divide the
real line into intervals which are either of positive or negative
type with respect to $AG$. Due to a theorem in \cite{lmm} this
implies the existence of local spectral functions of $AG$ on these
intervals. In the proof of Theorem \ref{t:sf} we "connect" those
local spectral functions and thus obtain a spectral function of $AG$
on $\R$ with a finite number of singularities.

The paper is arranged as follows. In the preliminaries section following this introduction we introduce the spectral points of positive and negative type of a symmetric operator in a (possibly indefinite) inner product space and prove that such an operator has a local spectral function on intervals of positive or negative type. In section \ref{s:products} we consider products $AG$ of selfadjoint operators $A$ and $G$ in a Hilbert space $(\calH,\hproduct)$ such that $\rho(AG)$ and $\rho(GA)$ are non-empty. The operator $AG$ is then symmetric with respect to the inner product $(G_0\cdot,\cdot)$, where $G_0$ is the bounded selfadjoint operator given by
$$
G_0 := G(AG - \la_0)^{-1}(AG - \ol{\la_0})^{-1},\quad\la_0\in\rho(AG)\setminus\R,
$$
and we analyze the spectra of positive and negative type of $AG$ (corresponding to the inner product $(G_0\cdot,\cdot)$). For example, it turns out that these spectra do not depend on the choice of $\la_0$. In section \ref{s:def} we particularly make use of the results in section \ref{s:products} to prove the main theorems on definitizable pairs of selfadjoint operators. In section \ref{s:application} we apply our results to Sturm-Liouville problems.

\section{Preliminaries}
Let $S$ be a linear operator in a Banach space $X$. If $S$ is bounded and everywhere defined, we write $S\in L(X)$. By the {\em resolvent set} $\rho(S)$ of $S$ we understand the set of all $\la\in\C$ for which $\ran(S - \la) = X$, $\ker(S - \la) = \{0\}$ and $(S - \la)^{-1}\in L(X)$. With this definition of $\rho(S)$, the operator $S$ is closed if $\rho(S)$ is non-empty. The operator $S$ is called {\it boundedly invertible} if $0\in\rho(S)$. The set $\sigma(S) := \C\setminus\rho(S)$ is called the {\em spectrum} of $S$. The {\it approximate point spectrum} $\sap(S)$ of $S$ is defined as the set of all $\la\in\C$ for which there exists a sequence $(x_n)\subset\dom S$ with $\|x_n\| = 1$ and $(S - \la)x_n\to 0$ as $n\to\infty$. A point $\la\in\C$ does not belong to $\sap(S)$ if and only if there exists $c > 0$ and an open neighborhood $\calU$ of $\la$ in $\C$ such that $\|(S - \mu)x\|\ge c\|x\|$ holds for all $x\in\dom S$ and all $\mu\in\calU$.

Throughout this section $(\calH,\hproduct)$ denotes a Hilbert space and $G_0$ a bounded selfadjoint operator in $\calH$. The operator $G_0$ induces a new inner product $\product$ on $\calH$ via
$$
[x,y] := (G_0x,y),\quad x,y\in\calH.
$$
The pair $(\calH,\product)$ is often referred to as a {\it $G_0$-space}. If $G_0$ is boundedly invertible, $(\calH,\product)$ is called a {\it Krein space}. A subspace $\calL$ of $\calH$ is called {\it uniformly positive} ({\it uniformly negative}) if there exists $\delta > 0$ such that
$$
[x,x]\ge\delta\|x\|^2\qquad\big([x,x]\le-\delta\|x\|^2,\text{ respectively}\big)
$$
holds for all $x\in\calL$. If $\calL$ is closed, then $\calL$ is uniformly positive (uniformly negative) if and only if $(\calL,\product)$ ($(\calL,-\product)$, respectively) is a Hilbert space. The {\it orthogonal companion} of a subspace $\calL$ is defined by
$$
\calL^\gperp := \{x\in\calH : [x,\ell] = 0\;\text{ for all }\,\ell\in\calL\}.
$$
The subspace $\calL$ is called {\it ortho-complemented} if $\calH = \calL\,+\,\calL^\gperp$. If the sum is direct, we write $\calH = \calL[\ds]\calL^\gperp$. The symbol $[\ds]$ thus denotes the direct $\product$-orthogonal sum. The following lemma will be used frequently, cf.\ \cite[Theorem 1.7.16]{ai}.

\begin{lem}\label{l:ks->oc}
Let $\calL\subset\calH$ be a closed subspace. If $(\calL,\product)$ is a Krein space, then $\calL$ is ortho-complemented. More precisely, we have
$$
\calH = \calL\,[\ds]\,\calL^\gperp.
$$
\end{lem}

A closed and densely defined linear operator $T$ in $\calH$ will be called {\it $G_0$-symmetric} (or $\product$-{\it symmetric}) if
$$
[Tx,y] = [x,Ty]\quad\text{holds for all }x,y\in\dom T.
$$
This is equivalent to the symmetry of the operator $G_0T$ in the Hilbert space $(\calH,\hproduct)$, i.e.\ $G_0T\subset (G_0T)^*$.

The Riesz-Dunford spectral projection of a closed linear operator $T$ in $\calH$ with respect to a spectral set $\sigma$ of $T$ will be denoted by $E(T;\sigma)$. If $\sigma = \{\la\}$, we write $E(T;\la)$ instead of $E(T;\{\la\})$.

\begin{lem}\label{l:E}
Let $T$ be $G_0$-symmetric. If $\la\in\C$ and $\ol\la$ are isolated points of the spectrum of $T$, we have
$$
[E(T;\la)x,y] = [x,E(T;\ol\la)y]\quad\text{for all }\;x,y\in\calH.
$$
\end{lem}
\begin{proof}
Let $\veps > 0$ be a number such that the deleted discs $\{\mu\in\C : |\mu - \la|\le\veps\}\setminus\{\la\}$ and $\{\mu\in\C : |\mu - \ol\la|\le\veps\}\setminus\{\ol\la\}$ are contained in $\rho(T)$. Define the curves $\gamma,\psi : [0,2\pi]\to\C$ by
$$
\gamma(t) := \la + \veps e^{it} \:\text{ and }\; \psi(t) := \ol\la + \veps e^{it}, \;\;t\in [0,2\pi].
$$
Then for $x,y\in\calH$ we have
\begin{eqnarray*}
[E(T;\la)x,y] &=& -\frac{1}{2\pi i}\int_\gamma [(T - \mu)^{-1}x,y]\,d\mu = \frac{1}{2\pi i}\int_{\gamma^{-1}} [x,(T - \ol\mu)^{-1}y]\,d\mu \\
              &=& \frac{1}{2\pi i}\int_0^{2\pi} [x,(T - \ol\la - \veps e^{it})^{-1}y](-i)\veps e^{-it}\,dt \\
              &=& \left[ x , -\frac{1}{2\pi i}\int_0^{2\pi} i\veps e^{it}(T - \psi(t))^{-1}y\,dt \right] = [x,E(T;\ol\la)y],
\end{eqnarray*}
where $\gamma^{-1}(t) := \la + \veps e^{-it}$, $t\in [0,2\pi]$.
\end{proof}

In \cite{lmm} the spectral points of positive and negative type of a bounded $G_0$-symmetric operator were introduced. In the following definition these notions are extended to unbounded operators.

\begin{defn}
Let $T$ be a $G_0$-symmetric operator in $\calH$. A point $\la\in\sap(T)$ is called a {\em spectral point of positive {\rm (}negative{\rm )} type} of $T$ if for every sequence $(x_n)\subset\dom T$ with $\|x_n\| = 1$ and $(T - \la)x_n\to 0$ as $n\to\infty$ we have
$$
\liminf_{n\to\infty}\,[x_n,x_n] > 0\quad\Big(\limsup_{n\to\infty}\,[x_n,x_n] < 0,\;\text{respectively}\Big).
$$
The set of all spectral points of positive {\rm (}negative{\rm )} type of $T$ will be denoted by $\sp(T)$ {\rm (}$\sm(T)$, respectively{\rm )}. A set $\Delta\subset\C$ is said to be of positive {\rm (}negative{\rm )} type with respect to $T$ if
$$
\Delta\cap\sap(T)\subset\sp(T)\quad\Big(\Delta\cap\sap(T)\subset\sm(T),\,\text{respectively}\Big).
$$
\end{defn}

The following statements were proved in \cite{lmm} for bounded operators. However, the proofs can be adopted without difficulties in the unbounded case.

\begin{prop}\label{p:def_type}
The spectral points of positive and negative type of a $G_0$-sym\-met\-ric operator $T$ are real. Moreover, $\sp(T)$ and $\sm(T)$ are open in $\sap(T)$. In particular, if $\Delta$ is a compact interval which is of positive {\rm (}negative{\rm )} type with respect to $T$, then there exists a $\C$-open neighborhood $\calU$ of $\Delta$ such that $(\calU\setminus\R)\cap\sap(T) = \emptyset$ and $\calU\cap\R$ is of positive type {\rm (}negative type, respectively{\rm )} with respect to $T$. Moreover, there exists $C > 0$ such that for all $\la\in\calU$ we have
$$
\|(T - \la)x\|\,\ge\,C|\Im\la|\,\|x\|,\quad x\in\dom T.
$$
\end{prop}

\begin{defn}\label{d:sf}
Let $J\subset\R$ be a bounded or unbounded open interval and let $s\subset J$ be a finite set. The system consisting of all bounded Borel subsets $\Delta$ of $J$ with $\ol\Delta\subset J$ the boundary points of which are not contained in $s$ will be denoted by $\mathfrak R_s(J)$. If $s = \emptyset$, we simply write $\mathfrak R(J)$. Let $S$ be a closed and densely defined linear operator in the Banach space $X$. A set function $E$ mapping from $\mathfrak R_s(J)$ into the set of bounded projections in $X$ is called a {\it local spectral function of $S$ on $J$} ({\it with the set of critical points $s = s(E)$}) if the following conditions are satisfied for all $\Delta,\Delta_1,\Delta_2\in\mathfrak R_s(J)$:
\begin{enumerate}
\item[{\rm (S1)}] $E(\Delta_1\cap\Delta_2) = E(\Delta_1)E(\Delta_2)$.
\item[{\rm (S2)}] If $\Delta_1\cap\Delta_2 = \emptyset$, then $E(\Delta_1\cup\Delta_2) = E(\Delta_1) + E(\Delta_2)$.
\item[{\rm (S3)}] $E(\Delta)$ commutes with every operator $B\in L(\calH)$ for which $BS\,\subset\,SB$.
\item[{\rm (S4)}] $\sigma(S|E(\Delta)\calH)\subset\ol{\sigma(S)\cap\Delta}$.
\item[{\rm (S5)}] $\sigma(S|(I - E(\Delta))\calH)\subset\ol{\sigma(S)\setminus\Delta}$.
\end{enumerate}
The points $\la\in s(E)$ for which the strong limits
$$
s-\lim_{t\to 0}E([\la-\veps,\la-t])\quad\text{ and }\quad s-\lim_{t\to 0}E([\la+t,\la+\veps])
$$
do not exist for sufficiently small $\veps > 0$ are called the {\it singularities of $E$}.
\end{defn}

\begin{rem}
Note that $BS\subset SB$ is equivalent to $B(S - \la)^{-1} = (S - \la)^{-1}B$ for every $\la\in\rho(S)$ if $\rho(S)\neq\emptyset$.
\end{rem}

Let $S$ be a closed operator in a Banach space and let $\Delta$ be a compact set in $\C$. A closed subspace $\calL_\Delta\,\subset\,\dom S$ is called the {\it maximal spectral subspace} of $S$ corresponding to $\Delta$ if the following holds:
\begin{enumerate}
\item[(a)] $S\calL_\Delta\subset\calL_\Delta$.
\item[(b)] $\sigma(S|\calL_\Delta)\subset\sigma(S)\cap\Delta$
\item[(c)] If $\calL\subset\dom S$ is a closed subspace such that (a) and (b) hold with $\calL_\Delta$ replaced by $\calL$ then $\calL\subset\calL_\Delta$.
\end{enumerate}
By $\C^+$ ($\C^-$) we denote the open upper (lower, respectively) halfplane. The following theorem has been shown for bounded $G_0$-symmetric operators in \cite{lmm}.

\begin{thm}\label{t:lsf}
Let $J$ be a bounded or unbounded open interval in $\R$ which is of positive {\rm (}negative{\rm )} type with respect to the $G_0$-symmetric operator $T$. If each of the sets $\C^+\cap\rho(T)$ and $\C^-\cap\rho(T)$ has an accumulation point in $J$, then $T$ has a local spectral function $E$ without critical points on $J$ with the following properties {\rm (}$\Delta\in\mathfrak R(J)${\rm )}:
\begin{enumerate}
\item[{\rm (i)}]   The subspace $E(\Delta)\calH$ is uniformly positive {\rm (}uniformly negative, respectively{\rm )}.
\item[{\rm (ii)}]  The operator $E(\Delta)$ is $G_0$-symmetric.
\item[{\rm (iii)}] If $\Delta$ is compact, then $E(\Delta)\calH$ is the maximal spectral subspace of $T$ corresponding to $\Delta$.
\end{enumerate}
\end{thm}
\begin{proof}
Let $J$ be of positive type with respect to $T$. As a consequence of the uniqueness of a local spectral function (see \cite[Lemma 3.14]{j2}) it is sufficient to prove that the operator $T$ has a local spectral function on each compact subinterval of $J$. Let $J'$ be such an interval. Choosing a larger compact interval which contains accumulation points of $\C^+\cap\rho(T)$ and $\C^-\cap\rho(T)$ it is seen from Proposition \ref{p:def_type} that there exists an open neighborhood $\calU$ of $J'$ in $\C$ such that $\calU\setminus\R\subset\rho(T)$ and that there exists $C > 0$ such that
\begin{equation}\label{e:growth}
\|(T - \la)^{-1}\|\,\le\,\frac C {|\Im\la|}
\end{equation}
holds for all $\la\in\calU\setminus\R$. By \cite[Chapter II, \paragraf 2, Theorem 5]{lm} the maximal spectral subspace $\calL$ of $T$ corresponding to $J'$ exists and $T|\calL$ is bounded. As $T|\calL$ is also $\product$-symmetric and $\sigma(T|\calL) = \sp(T|\calL)$ it follows from \cite[Theorem 3.1]{lmm} that $(\calL,\product)$ is a Hilbert space. Denote by $E_\calL$ the spectral measure of the selfadjoint operator $T|\calL$ in $(\calL,\product)$ and by $P_\calL$ the projection onto $\calL$ with $\ker P_\calL = \calL^\gperp$ which exists due to Lemma \ref{l:ks->oc}. Then $E(\cdot) := E_\calL(\cdot)P_\calL$ defines a local spectral function of $T$ on $J'$.
\end{proof}

\section{Products of selfadjoint operators}\label{s:products}
Throughout this section let $A$ and $G$ be (possibly unbounded and/or non-invertible) selfadjoint operators in the Hilbert space $(\calH,\hproduct)$. Each of the statements in the following proposition follows from or is an easy consequence of \cite[Remark 2.5]{hkm} and \cite[Theorem 1.1]{hkm}, see also \cite{hm}.

\begin{prop}\label{p:basic}
Let $A$ and $G$ be selfadjoint operators in $\calH$. If
\begin{equation}\label{e:ass}\tag{$*$}
\rho(AG)\neq\emptyset\quad\text{and}\quad\rho(GA)\neq\emptyset,
\end{equation}
then both operators $AG$ and $GA$ are closed and densely defined and
\begin{equation}\label{e:adj}
(AG)^* = GA.
\end{equation}
Moreover,
\begin{equation}\label{e:sigs}
\sigma(AG)\setminus\{0\} = \sigma(GA)\setminus\{0\}.
\end{equation}
In addition, for $\la\in\rho(AG)\setminus\{0\}$ the following relations hold:
\begin{align*}
A(GA - \la)^{-1} &= \ol{(AG - \la)^{-1}A}\\
G(AG - \la)^{-1} &= \ol{(GA - \la)^{-1}G}.
\end{align*}
\end{prop}

In our main results (Theorems \ref{t:main} and \ref{t:sf} below) we require that \eqref{e:ass} is satisfied. Since in applications this condition might be hard to verify, the following sufficient conditions for \eqref{e:ass} may be helpful.

\begin{lem}
The following conditions are sufficient for {\rm\eqref{e:ass}} to hold:
\begin{enumerate}
\item[{\rm (a)}] $G$ is bounded and $\rho(GA)\neq\emptyset$.
\item[{\rm (b)}] $G$ is boundedly invertible and $\rho(AG)\neq\emptyset$.
\item[{\rm (c)}] $(AG)^* = GA$ and $\rho(AG)\neq\emptyset$.
\item[{\rm (d)}] $\rho(AG)\neq\emptyset$, $GA$ is closed and for some $\la\in\rho(AG)\setminus\{0\}$ the operator $G(AG - \la)^{-1}A$ is bounded on $\dom A$.
\end{enumerate}
\end{lem}
\begin{proof}
If (c) holds, then $\sigma(GA) = \sigma((AG)^*) = \{\ol\la : \la\in\sigma(AG)\}$ and hence $\rho(GA)\neq\emptyset$. If (b) holds, then $AG$ and $GA$ are closed and $(AG)^* = GA$. If (a) holds, then $AG$ and $GA$ are closed and $(GA)^* = AG$. Hence, in both cases \eqref{e:ass} follows from (c). Assume now that (d) holds. Then the operator $GA - \la$ is injective. Moreover, for $x\in\dom A$ we have $G(AG - \la)^{-1}Ax  - x \in\dom(GA)$ and
$$
(GA - \la)\big(G(AG - \la)^{-1}Ax  - x\big) = \la x.
$$
This shows that $\dom A\subset\ran(GA - \la)$ and that $(GA - \la)^{-1}|\dom A$ is bounded. As the closure of $(GA - \la)^{-1}|\dom A$ coincides with $(GA - \la)^{-1}$ (on $\ran(GA - \la)$), it follows that $(GA - \la)^{-1}\in L(\calH)$.
\end{proof}

\begin{rem}\label{r:no_star}
If $AG\in L(\calH)$ or $GA\in L(\calH)$ then either $A,G\in L(\calH)$ or ($*$) does not hold.
\end{rem}

Indeed, if $AG\in L(\calH)$, then $\dom G = \calH$ yields $G\in L(\calH)$. Suppose that ($*$) holds. Then, according to Proposition \ref{p:basic}, we have $GA = (AG)^*\in L(\calH)$ and thus $A\in L(\calH)$.

\begin{prop}\label{p:zero_res}
If \eqref{e:ass} is satisfied, then the following statements are equivalent.
\begin{enumerate}
\item[{\rm (a)}] $AG$ is boundedly invertible.
\item[{\rm (b)}] $\ran(AG) = \calH$.
\item[{\rm (c)}] $GA$ is boundedly invertible.
\item[{\rm (d)}] $\ran(GA) = \calH$.
\item[{\rm (e)}] $A$ and $G$ are boundedly invertible.
\end{enumerate}
In particular, $\sigma(AG) = \sigma(GA)$.
\end{prop}
\begin{proof}
Clearly, (a) implies (b). Assume that (b) holds. Then $\ran A = \calH$ (which implies $\ker A = \{0\}$) and $\dom A\subset\ran G$ (which implies $\ker G = (\ran G)^\perp\subset (\dom A)^\perp = \{0\}$). Hence, $\ker(AG) = \{0\}$ and (a) follows. An analog reasoning shows that (c) holds if and only if (d) holds. The equivalence (a)$\Slra$(c) is a consequence of \eqref{e:adj}. Since (a) implies that $A$ is boundedly invertible, (c) implies that $G$ is boundedly invertible and (e) implies both (a) and (c), the proposition is proved.
\end{proof}

\begin{cor}\label{c:sigs}
Assume that \eqref{e:ass} holds. Then for each $\la\in\C$ the following statements hold.
\begin{enumerate}
\item[{\rm (i)}]   $\la\in\sigma(AG)\;\Llra\;\ol\la\in\sigma(AG)$.
\item[{\rm (ii)}]  $\la\in\sigma(AG)\setminus\sap(AG)\;\Lra\;\ol\la\in\sigma_p(AG)$.
\end{enumerate}
\end{cor}
\begin{proof}
From Propositions \ref{p:basic} and \ref{p:zero_res} it follows that $\la\in\rho(AG)$ implies
$$
\ol\la\in\rho((AG)^*) = \rho(GA) = \rho(AG).
$$
This proves (i). Let us prove (ii) for $\la\neq 0$. If $\la\in\sigma(AG)\setminus\sap(AG)$, $\la\neq 0$, then it is well-known that $\ol\la\in\sigma_p((AG)^*) = \sigma_p(GA)$. Hence, there exists $x\in\dom(GA)\setminus\{0\}$ such that $GAx = \ol\la x$. Therefore, $GAx\in\dom A$ and $(AG - \ol\la)Ax = A(GA - \ol\la)x = 0$. Since $Ax\neq 0$ (otherwise, $GAx = 0$ and thus $x=0$), we conclude that $\ol\la\in\sigma_p(AG)$. But (ii) also holds for $\la=0$ as in this case the left hand side of the implication (ii) is never true. To see this, note that $0\notin\sap(AG)$ implies that there is a neighborhood $\calU$ of zero such that $\calU\cap\sap(AG) = \emptyset$. Now, from (ii) for $\la\neq 0$ it follows that $\calU\setminus\{0\}\subset\rho(AG)$. Hence, the Fredholm index of $AG - \la$ for $\la\in\calU$ is constantly zero. And as $\ker(AG) = \{0\}$, it follows that also $0\in\rho(AG)$.
\end{proof}

If \eqref{e:ass} is satisfied, by Corollary \ref{c:sigs} there exists $\la_0\in\C\setminus\R$ such that $\la_0,\ol{\la_0}\in\rho(AG)$, and thus, the operator
\begin{equation}\label{e:G0}
G_0 := G(AG - \la_0)^{-1}(AG - \ol{\la_0})^{-1}
\end{equation}
is bounded. Moreover, due to Proposition \ref{p:basic} we have
\begin{align*}
G_0^*
&= (GA - \la_0)^{-1}\big(G(AG - \la_0)^{-1}\big)^*\\
&= (GA - \la_0)^{-1}\big((GA - \la_0)^{-1}G\big)^*\\
&= (GA - \la_0)^{-1}G(AG - \ol{\la_0})^{-1}\\
&= G(AG - \la_0)^{-1}(AG - \ol{\la_0})^{-1} = G_0
\end{align*}
and
\begin{align*}
G_0AG
&= G(AG - \la_0)^{-1}(AG - \ol{\la_0})^{-1}AG\\
&\subset GAG(AG - \la_0)^{-1}(AG - \ol{\la_0})^{-1}\\
&= GAG_0 = (G_0AG)^*.
\end{align*}
This shows that $G_0$ is selfadjoint and that $AG$ is $G_0$-symmetric. Equivalently, $AG$ is symmetric with respect to the inner product
\begin{equation}\label{e:ip}
[x,y] := (G_0x,y),\quad x,y\in\calH.
\end{equation}
Note that the inner product $\product$ is in general not a Krein space inner product. It might even be degenerate.

For the rest of this section we assume that \eqref{e:ass} holds and fix $\la_0\in\rho(AG)\setminus\R$, the operator $G_0$ in \eqref{e:G0} and the inner product $\product$ in \eqref{e:ip}. The spectra of positive and negative type of $AG$ are connected with the inner product $\product$ which itself depends on $\la_0\in\rho(AG)\setminus\R$. The following lemma shows that $\sp(AG)$ and $\sm(AG)$ are in fact independent of $\la_0$.

\begin{lem}\label{l:spsp}
Let $\la\in\C$. Then $\la\in\sp(AG)$ {\rm (}$\la\in\sm(AG)${\rm )} if and only if for each sequence $(x_n)\subset\dom AG$ with $\|x_n\| = 1$ and $(AG - \la)x_n\to 0$ as $n\to\infty$ we have
$$
\liminf_{n\to\infty}\,(Gx_n,x_n) > 0\quad\left(\limsup_{n\to\infty}\,(Gx_n,x_n) < 0,\text{ respectively}\right).
$$
\end{lem}
\begin{proof}
Assume that the condition in the lemma on the approximate eigensequences of $AG$ holds and let $(x_n)\subset\dom AG$ with $\|x_n\| = 1$ and $(AG - \la)x_n\to 0$ as $n\to\infty$. Set
$$
y_n := (\la - \la_0)(AG - \la_0)^{-1}x_n.
$$
Then we have
$$
(AG - \la)y_n = (\la - \la_0)(x_n + (\la_0 - \la)(AG - \la_0)^{-1}x_n) = (\la - \la_0)(x_n - y_n).
$$
On the other hand,
$$
(AG - \la)y_n = (\la - \la_0)(AG - \la_0)^{-1}(AG - \la)x_n\,\lra\,0
$$
as $n\to\infty$. Hence $\|y_n\|\to 1$ and since
$$
[x_n,x_n] = (G(AG - \la_0)^{-1}x_n,(AG - \la_0)^{-1}x_n) = \frac{1}{|\la - \la_0|^2}(Gy_n,y_n)
$$
we conclude
$$
\liminf_{n\to\infty}\,[x_n,x_n] = \frac{1}{|\la - \la_0|^2}\,\liminf_{n\to\infty}\,(Gy_n,y_n) > 0.
$$
Conversely, let $\la\in\sp(AG)$ and let $(x_n)\subset\dom AG$ with $\|x_n\| = 1$ and $(AG - \la)x_n\to 0$ as $n\to\infty$. Since
$$
(Gx_n,x_n) = [(AG - \la_0)x_n,(AG - \la_0)x_n],
$$
we obtain from $(AG - \la)x_n\to 0$ as $n\to\infty$:
$$
\liminf_{n\to\infty}\,(Gx_n,x_n) = |\la - \la_0|^2\,\liminf_{n\to\infty}\,[x_n,x_n] > 0,
$$
which proves the assertion.
\end{proof}

\begin{cor}\label{c:zero_+}
Assume that \eqref{e:ass} holds and that $0\in\sp(AG)\cup\sm(AG)$. Then $G$ is boundedly invertible.
\end{cor}
\begin{proof}
Suppose that, e.g., $0\in\sp(AG)$ and that there exists a sequence $(x_n)\subset\dom G$ with $\|x_n\|=1$ for $n\in\N$ and $Gx_n\to 0$ as $n\to\infty$. Define
$$
y_n := -\la_0(AG - \la_0)^{-1}x_n\,\in\,\dom(GAG)
$$
as in the proof of Lemma \ref{l:spsp} (with $\la = 0$). Then $AGy_n = \la_0(y_n - x_n)$ and $AGy_n = -\la_0A(GA - \la_0)^{-1}Gx_n\to 0$ as $n\to\infty$ as $A(GA - \la_0)^{-1}$ is bounded. Therefore, $\|y_n\|\to 1$ and since $0\in\sp(AG)$, we conclude $\liminf_{n\to\infty}\,(Gy_n,y_n) > 0$ from Lemma \ref{l:spsp}. But this contradicts $Gy_n = -\la_0(GA - \la_0)^{-1}Gx_n\to 0$ as $n\to\infty$.
\end{proof}

\begin{lem}\label{l:krein}
Assume that \eqref{e:ass} is satisfied. Let $\calL\subset\dom AG$ be a closed subspace such that $AG\calL\subset\calL$ and $0\in\rho(AG|\calL)$. If $\calH = \calL + \calL^\gperp$, then $(\calL,\product)$ is a Krein space.
\end{lem}
\begin{proof}
Let $P_\calL$ be the orthogonal projection (with respect to $\hproduct$) onto $\calL$ in $\calH$. Then, with $G_\calL := P_\calL (G_0|\calL)\in L(\calL)$ we have
$$
[\ell_1,\ell_2] = (G_\calL\ell_1,\ell_2) \quad\text{for }\;\ell_1,\ell_2\in\calL.
$$
Hence, $(\calL,\product)$ is a Krein space if and only if $G_\calL$ is boundedly invertible. Let $\ell\in\ker G_\calL$. By assumption, for any $x\in\calH$ we find $x_1\in\calL$ and $x_2\in\calL^\gperp$ such that $x = x_1 + x_2$. It follows that
$$
(G_0\ell,x) = [\ell,x_1+x_2] = [\ell,x_1] = (G_\calL\ell,x_1) = 0,
$$
and thus $G_0\ell = 0$. From
$$
0 = G(AG - \la_0)^{-1}(AG - \ol{\la_0})^{-1}\ell = (GA - \la_0)^{-1}(GA - \ol{\la_0})^{-1}G\ell
$$
we conclude $G\ell = 0$ and hence $AG\ell = 0$ which implies $\ell = 0$ as $0\in\rho(AG|\calL)$. Therefore we have $\calH = \calL [\ds] \calL^\gperp$ (since $\ker G_\calL = \calL\cap\calL^\gperp$).

Now, suppose that there exists a sequence $(\ell_n)\subset\calL$ with $\|\ell_n\|=1$ and $\|G_\calL\ell_n\| \to 0$ as $n\to\infty$. If by $P$ we denote the ($G_0$-symmetric) projection onto $\calL$ with $\ker P = \calL^\gperp$, we obtain
\begin{align*}
\|G_0\ell_n\|^2
&= (G_0\ell_n,PG_0\ell_n) + (G_0\ell_n,(I - P)G_0\ell_n)\\
&= (P_\calL G_0\ell_n, PG_0\ell_n) + [\ell_n, (I - P)G_0\ell_n]\\
&= (G_\calL\ell_n,PG_0\ell_n)\\
&\le \|G_\calL \ell_n\|\cdot\|P\|\cdot\|G_0\|.
\end{align*}
Hence, $G_0\ell_n\to 0$ as $n\to\infty$. It is easy to see that $\calL^\gperp$ is $AG$-invarant. Hence, $\calL$ is $(AG - \la_0)^{-1}$-invariant. And since $AG|\calL$ is bounded, we conclude
\begin{align*}
\|AG_0\ell_n\|
&\le \|(AG - \la_0)|\calL\|\cdot\|(AG - \la_0)^{-1}AG_0\ell_n\|\\
&= \|(AG - \la_0)|\calL\|\cdot\|A(GA - \la_0)^{-1}G_0\ell_n\|\\
&\le \|(AG - \la_0)|\calL\|\cdot\|A(GA - \la_0)^{-1}\|\cdot\|G_0\ell_n\|.
\end{align*}
Thus, we have $(AG - \la_0)^{-1}(AG - \ol{\la_0})^{-1}AG\ell_n = AG_0\ell_n\to 0$, which implies $AG\ell_n\to 0$ as $n\to\infty$, which is a contradiction to $0\in\rho(AG|\calL)$. The lemma is proved.
\end{proof}

\begin{prop}\label{p:iso}
Assume that \eqref{e:ass} is satisfied. Then for each $\la\in\C$ the following statements hold.
\begin{enumerate}
\item[{\rm (i)}]  If $\la\neq 0$ is an isolated point of the spectrum of $AG$ {\rm (}and hence also $\ol\la${\rm )}, then the inner product space $(E(AG;\{\la,\ol\la\})\calH,\product)$ is a Krein space.
\item[{\rm (ii)}] If $\la$ is a pole of the resolvent of $AG$ of order $\nu$ then $\ol\la$ is a pole of the resolvent of $AG$ of order $\nu$.
\end{enumerate}
\end{prop}
\begin{proof}
For the proof of (i) set $E := E(AG;\{\la,\ol\la\})$. As $E$ is $\product$-symmetric by Lemma \ref{l:E}, it follows that $(I - E)\calH\subset (E\calH)^\gperp$. And since $\calH = E\calH\ds(I - E)\calH$, Lemma \ref{l:krein} yields the assertion.

By \cite[Theorem VII.3.18]{ds} the fact that $\la\notin\R$ (the statement for $\la\in\R$ is trivial) is a pole of the resolvent of $AG$ of order $\nu$ is equivalent to
$$
(AG - \la)^\nu E(AG;\la) = 0 \;\text{ and }\; (AG - \la)^{\nu - 1}E(AG;\la)\neq 0.
$$
Let $x,v\in E(AG;\ol\la)\calH$ be arbitrary. From Lemma \ref{l:E} we obtain
\begin{align*}
[(AG - \ol\la)^\nu x,v]
&= [E(AG;\ol\la)(AG - \ol\la)^\nu x, v]\\
&= [(AG - \ol\la)^\nu x,E(AG;\la)v] = 0.
\end{align*}
Furthermore, for $u\in E(AG;\la)\calH$ we have
$$
[(AG - \ol\la)^\nu x,u] = [x,(AG - \la)^\nu u] = 0.
$$
Hence, $[(AG - \ol\la)^\nu x,y] = 0$ for all $y\in E(AG;\{\la,\ol\la\})\calH$. But $(E(AG;\{\la,\ol\la\})\calH,\product)$ is a Krein space by (i), and we obtain $(AG - \ol\la)^\nu x = 0$.
\end{proof}

\begin{prop}\label{p:bounded_non-negative}
Let $A_0$ be a bounded selfadjoint operator in $\calH$ and assume that $G_0A_0G_0\ge 0$. Then the following statements hold for the bounded  $G_0$-symmetric operator $A_0G_0$:
\begin{enumerate}
\item[{\rm (i)}]   $\sigma(A_0G_0)\subset\R$\,,
\item[{\rm (ii)}]  $(0,\infty)\cap\sigma(A_0G_0)\subset\sp(A_0G_0)$,
\item[{\rm (iii)}] $(-\infty,0)\cap\sigma(A_0G_0)\subset\sm(A_0G_0)$.
\end{enumerate}
\end{prop}
\begin{proof}
Let $\la\in\sap(A_0G_0)\setminus\{0\}$ and let $(x_n)\subset\calH$ with $\|x_n\| = 1$, $n\in\N$, and $(A_0G_0 - \la)x_n\to 0$ as $n\to\infty$. We claim that it is not possible that $\lim_{n\to\infty}\,(G_0A_0G_0x_n,x_n) = 0$. Suppose the contrary. Then, from the Cauchy-Bun\-ya\-kowski inequality we obtain
\begin{align*}
\|G_0A_0G_0x_n\|^2 \le (G_0A_0G_0x_n,x_n)((G_0A_0G_0)^2x_n,G_0A_0G_0x_n),
\end{align*}
and hence $G_0A_0G_0x_n\to 0$ as $n\to\infty$. As $(A_0G_0 - \la)x_n\to 0$, this implies $G_0x_n\to 0$ and hence $A_0G_0x_n\to 0$ as $n\to\infty$. A contradiction.

Assume that there exists $\la\in\sap(A_0G_0)\setminus\R$. Then there exists $(x_n)\subset\calH$ with $\|x_n\| = 1$ and $(A_0G_0 - \la)x_n\to 0$ as $n\to\infty$. Since $[A_0G_0x_n,x_n] - \la [x_n,x_n]$ tends to zero as $n\to\infty$ and $[A_0G_0x_n,x_n]$ and $[x_n,x_n]$ both are real for each $n$, it follows from $\la\notin\R$ that $[A_0G_0x_n,x_n]$ tends to zero which contradicts the statement proved above. Hence $\sap(A_0G_0)\setminus\R = \emptyset$, and from Corollary \ref{c:sigs}(ii) we obtain $\sigma(A_0G_0)\subset\R$.

Let $\la\in\sigma(A_0G_0)$, $\la > 0$. Then $\la\in\sap(A_0G_0)$ by Corollary \ref{c:sigs}(ii). Let $(x_n)\subset\calH$ with $\|x_n\| = 1$ and $(A_0G_0 - \la)x_n\to 0$ as $n\to\infty$. Suppose $\liminf_{n\to\infty}\,[x_n,x_n]\le 0$. Then from
$$
\la\liminf_{n\to\infty}\,[x_n,x_n] = \liminf_{n\to\infty}\,[(\la - A_0G_0)x_n,x_n] + (G_0A_0G_0x_n,x_n)\ge 0
$$
it is seen that there exists a subsequence $(x_{n_k})$ such that $(G_0A_0G_0x_{n_k},x_{n_k})$ tends to zero as $k\to\infty$. But this is a contradiction to the statement proved above, and it follows that
$$
\liminf_{n\to\infty}\,[x_n,x_n] > 0.
$$
This shows (ii), and (iii) can be shown similarly.
\end{proof}

\begin{rem}
In the above proof the special representation $G_0 = G(AG - \la_0)^{-1}(AG - \ol{\la_0})^{-1}$ of $G_0$ was not used. Therefore, Proposition \ref{p:bounded_non-negative} also holds for arbitrary bounded selfadjoint operators $G_0$ in $\calH$.
\end{rem}

As a corollary of Proposition \ref{p:bounded_non-negative} we give another proof of a theorem of Radjavi and Rosenthal (see \cite[Proposition 6.8]{rr}). Recall that a closed subspace is hyperinvariant for $T\in L(X)$, $X$ a Banach space, if it is invariant for any operator in $L(X)$ which commutes with $T$.

\begin{cor}
Let $S,T\in L(\calH)$ be selfadjoint such that $STS\ge 0$. If $TS$ is not a constant multiple of the identity, then $TS$ has a non-trivial hyperinvariant subspace.
\end{cor}
\begin{proof}
If $\sigma(TS)\neq\{0\}$, then the assertion follows from Proposition \ref{p:bounded_non-negative} and Theorem \ref{t:lsf} (note that a maximal spectral subspace is hyperinvariant, cf.\ \cite[Proposition 2.3.2]{cf}). Hence, suppose that $\sigma(TS) = \{0\}$. It is no restriction to assume that $S$ and $T$ are injective. Otherwise, $\ker(TS)$ or $\ol{\ran TS} = \ker(ST)^\perp$ is hyperinvariant for $TS$ or $TS = 0$. Hence, $T$ is a non-negative operator and Proposition \ref{p:basic} yields $\sigma(T^{1/2}ST^{1/2}) = \{0\}$. But $T^{1/2}ST^{1/2}$ is selfadjoint and thus coincides with the zero operator. This yields $T = S = 0$, a contradiction.
\end{proof}

\section{Definitizable pairs of selfadjoint operators}\label{s:def}
In the following we extend the notion of definitizability of selfadjoint operators in Krein spaces to products (or pairs) of selfadjoint operators in a Hilbert space. As in the previous section let $A$ and $G$ be selfadjoint operators in the Hilbert space $(\calH,\hproduct)$. Again, if \eqref{e:ass} is satisfied for $A$ and $G$ we fix $\la_0\in\rho(AG)$, define the bounded selfadjoint operator $G_0$ as in \eqref{e:G0} and set $[\cdot,\cdot] := (G_0\cdot,\cdot)$.

\begin{defn}\label{d:def}
An ordered pair $(A,G)$ of selfadjoint operators is called {\em definitizable} if the resolvent sets of $AG$ and $GA$ are non-empty and if there exists a polynomial $p\neq 0$ with real coefficients such that
$$
(p(AG)x,Gx)\ge 0\quad\text{for all }x\in\dom(AG)^{\max\{1,d\}},
$$
where $d := \deg(p)$. The polynomial $p$ is called {\em definitizing} for $(A,G)$.
\end{defn}

If $G$ is bounded and boundedly invertible, then $AG$ is selfadjoint in the Krein space $(\calH,(G\cdot,\cdot))$ and Definition \ref{d:def} coincides with the definition of definitizability of the operator $AG$ in this Krein space. The next lemma shows that the definitizability of $(A,G)$ can also be expressed by means of the inner product $\product$.

\begin{lem}\label{l:indep}
Assume that \eqref{e:ass} is satisfied. Let $p\neq 0$ be a polynomial with real coefficients. Then the following statements are equivalent.
\begin{enumerate}
\item[{\rm (i)}]  $(A,G)$ is definitizable with definitizing polynomial $p$.
\item[{\rm (ii)}] $[p(AG)x,x]\ge 0$ holds for all $x\in\dom p(AG)$.
\end{enumerate}
\end{lem}
\begin{proof}
Let $d$ be the degree of $p$. If (i) holds and $y\in\dom(AG)^d$, then with $x := (AG - \la_0)^{-1}y\in\dom(AG)^{d+1}$ we have
\begin{align*}
[p(AG)y,y]
&= (p(AG)(AG - \la_0)x,G_0(AG - \la_0)x)\\
&= ((AG - \la_0)p(AG)x,(GA - \ol{\la_0})^{-1}Gx)\\
&= (p(AG)x,Gx)\ge 0.
\end{align*}
Conversely, assume that (ii) holds and let $x\in\dom(AG)^{d+1}$. Then with $y := (AG - \la_0)x\in\dom(AG)^d$ the following holds:
\begin{align*}
(p(AG)x,Gx)
&= (p(AG)(AG - \la_0)^{-1}y,G(AG - \la_0)^{-1}y)\\
&= (p(AG)y,(GA - \ol{\la_0})^{-1}G(AG - \la_0)^{-1}y)\\
&= (p(AG)y,G_0y) = [p(AG)y,y]\ge 0.
\end{align*}
Hence, the proof is finished if $d = 0$. Let $d > 0$ and $x\in\dom(AG)^d$. As $\rho(AG)\neq\emptyset$, there exists a sequence $(x_n)\subset\dom(AG)^{d+1}$ such that for $k=0,1,\ldots,d$ we have $(AG)^kx_n\to(AG)^kx$ as $n\to\infty$. Moreover, due to $\dom AG\subset\dom G$ and the closedness of $AG$ and $G$ there exists $c > 0$ such that
$$
\|Gu\|\,\le\,c\big(\|u\| + \|AGu\|\big)\quad\text{for all }u\in\dom AG.
$$
Therefore, from $x_n\to x$ and $AGx_n\to AGx$ we conclude $Gx_n\to Gx$ as $n\to\infty$. This gives $(p(AG)x,Gx) = \lim_{n\to\infty}\,(p(AG)x_n,Gx_n)\ge 0$. The lemma is proved.
\end{proof}

The proof of the following lemma is similar to that of Lemma \ref{l:indep} and is therefore omitted.

\begin{lem}
Let $p\neq 0$ be a polynomial with real coefficients and degree $d$. Then the following holds:
\begin{enumerate}
\item[{\rm (a)}] If $(A,G)$ is definitizable with definitizing polynomial $p$, then $(G,A)$ is definitizable with definitizing polynomial $\la p(\la)$.
\item[{\rm (b)}] If $G$ is boundedly invertible, then $(A,G)$ is definitizable with definitizing polynomial $p$ if and only if the relation $(p(GA)x,G^{-1}x)\ge 0$ holds for all $x\in\dom(GA)^{\max\{1,d\}}$.
\end{enumerate}
\end{lem}

It is well-known (see \cite{l}) that the spectrum of a definitizable operator $T$ in a Krein space is real -- with the possible exception of a finite number of non-real poles of the resolvent of $T$ -- and that $T$ has a spectral function on $\R$ with a finite number of singularities. The following two theorems generalize this result to definitizable pairs of selfadjoint operators.

\begin{thm}\label{t:main}
If $(A,G)$ is definitizable, then the following statements hold.
\begin{enumerate}
\item[{\rm (a)}] The non-real spectrum of $AG$ consists of a finite number of points which are poles of the resolvent of $AG$. Each such point is a zero of every definitizing polynomial for $(A,G)$.
\item[{\rm (b)}] If $\la\in\sigma(AG)\cap(\R\setminus\{0\})$ and $p(\la) > 0$ for some definitizing polynomial $p$ for $(A,G)$, then $\la\in\sp(AG)$.
\item[{\rm (c)}] If $\la\in\sigma(AG)\cap(\R\setminus\{0\})$ and $p(\la) < 0$ for some definitizing polynomial $p$ for $(A,G)$, then $\la\in\sm(AG)$.
\end{enumerate}
\end{thm}
\begin{proof}
Let $p$ be a definitizing polynomial for $(A,G)$ and set $m := \deg(p)+1$. Let $z_0\in\C\setminus\R$ such that $p(z_0)\neq 0$. First of all let us prove that there exists some $\la_1\in\rho(AG)$ such that
$$
z_0^2p(z_0)(z_0 - \la_1)^{-m-1}(z_0 - \ol{\la_1})^{-m-1}\,\notin\,\R.
$$
To see this, choose two open intervals $J_1$ and $J_2$ such that $0\notin J_2$, $z_0\notin J_1\times J_2$ and $J_1\times J_2\subset\rho(AG)$. With $\la = x + iy\in J_1\times J_2$ and $z_0 = \alpha_0 + i\beta_0$ we have
$$
(z_0 - \la)(z_0 - \ol\la) = (\alpha_0 - x)^2 - \beta_0^2 + y^2 + 2i\beta_0(\alpha_0 - x) =: f(x,y).
$$
The function $f : J_1\times J_2\to\R^2$ has the derivative
$$
f'(x,y) = \mat{-2(\alpha_0 - x)}{2y}{-2\beta_0}{0}.
$$
Its determinant equals $4\beta_0y$ and does therefore not vanish as $0\notin J_2$ and $z_0\notin\R$. Hence, $f(J_1\times J_2)$ is an open set in $\C\setminus\{0\}$, and thus also
$$
\{z_0^2p(z_0)(z_0 - \la)^{-m-1}(z_0 - \ol\la)^{-m-1} : \la\in J_1\times J_2\} = \{z_0^2p(z_0)z^{-m-1} : z\in f(J_1\times J_2)\}
$$
is open.

By Lemma \ref{l:indep} it is no restriction to assume $\la_0 = \la_1\,(\neq z_0)$. For $k=1,2$ define the rational functions
\begin{equation}\label{e:r}
r_k(\la) := \la^2p(\la)(\la - \la_0)^{-m-k}(\la - \ol{\la_0})^{-m-k}.
\end{equation}
Then $r_1(z_0)\notin\R$. Define the bounded operator
\begin{equation}\label{e:A0}
A_0 := AGAp(GA)(GA - \la_0)^{-m}(GA - \ol{\la_0})^{-m}.
\end{equation}
It is not difficult to see that $A_0$ is selfadjoint. Moreover, we observe that
\begin{align*}
Gr_2(AG)
&= GAGAGp(AG)(AG - \la_0)^{-m-2}(AG - \ol{\la_0})^{-m-2}\\
&= G_0AGp(AG)AG(AG - \la_0)^{-m-1}(AG - \ol{\la_0})^{-m-1}\\
&= G_0A_0G_0.
\end{align*}
Similarly, one proves that
$$
r_1(AG) = A_0G_0.
$$
In addition, $G_0A_0G_0\ge 0$ holds as for $x\in\calH$ we have
$$
y := AG(AG - \la_0)^{-m-1}x\in\dom p(AG)
$$
and
\begin{align*}
(G_0A_0G_0x,x)
&= (GAGAGp(AG)(AG - \la_0)^{-m-2}(AG - \ol{\la_0})^{-m-2}x,x)\\
&= (GAG(AG - \ol{\la_0})^{-m-2}(AG - \la_0)^{-1}p(AG)y,x)\\
&= (GA(AG - \ol{\la_0})^{-m-1}G_0p(AG)y,x)\\
&= (G_0p(AG)y,y) = [p(AG)y,y]\ge 0.
\end{align*}
By virtue of Proposition \ref{p:bounded_non-negative} we obtain $\sigma(r_1(AG)) = \sigma(A_0G_0)\subset\R$. And since $r_1(\cdot)$ is analytic in a neighborhood of $\sigma(AG)\cup\{\infty\}$, it is a consequence of the spectral mapping theorem \cite[Theorem VII.9.5]{ds} that $r_1(\sigma(AG))\subset\R$ and thus $z_0\in\rho(AG)$. To complete the proof of (a) it remains to show that each $\la\in\sigma(AG)\setminus\R$ is a pole of the resolvent of $AG$. To this end we show that
\begin{equation}\label{e:delete}
p(AG)E(AG;\{\la,\ol\la\}) = 0.
\end{equation}
From this it follows that also $p(AG)E(AG;\la) = 0$. And since the spectrum of $AG|E(AG;\la)\calH$ coincides with $\{\la\}$, we have $(AG - \la)^\alpha E(AG;\la) = 0$, where $\alpha$ is the order of $\la$ as a zero of $p$. This and \cite[Theorem VII.3.18]{ds} imply the assertion. So, let us prove \eqref{e:delete}. Let $y\in E(AG;\la)\calH$ and $z\in E(AG;\ol\la)\calH$ be arbitrary. By Lemma \ref{l:E} we have $[p(AG)y,y] = [E(AG;\la)p(AG)y,y] = [p(AG)y,E(AG;\ol\la)y] = 0$, $[p(AG)z,z] = 0$ and thus
\begin{align*}
[p(AG)y,z] + [p(AG)z,y]
&= [p(AG)y,y+z] + [p(AG)z,y+z]\\
&= [p(AG)(y+z),y+z]\ge 0.
\end{align*}
But at the same time,
\begin{align*}
-[p(AG)y,z] - [p(AG)z,y]
&= [p(AG)y,-z] + [p(AG)(-z),y]\\
&= [p(AG)(y-z),y-z]\ge 0.
\end{align*}
Hence, $[p(AG)(y+z),y+z] = 0$ and thus $[p(AG)x,x] = 0$ holds for all $x\in E(AG;\{\la,\ol\la\})\calH$. By polarization we obtain $[p(AG)x,y] = 0$ for all vectors $x,y\in E(AG;\{\la,\ol\la\})\calH$. But $(E(AG;\{\la,\ol\la\})\calH,\product)$ is a Krein space by Proposition \ref{p:iso}(i), and $p(AG)x = 0$ for all $x\in E(AG;\{\la,\ol\la\})\calH$ follows. Hence, (a) is proved.

For the proof of (b) we observe that by (a) there exists a definitizing polynomial $p$ for $(A,G)$ such that $p(\la_0)\neq 0$. Define the rational function $r_1$ as in \eqref{e:r}. Let $\la_1\in\R\setminus\{0\}$ such that $p(\la_1) > 0$. Then also $r_1(\la_1) > 0$, and there exists a function $g$ which is analytic on $\calU := \ol\C\setminus\{\la_0,\ol{\la_0}\}$ such that
$$
r_1(\la) - r_1(\la_1) = g(\la)(\la - \la_1),\quad\la\in\calU.
$$
It is obvious that $g$ is a rational function with the poles $\la_0$ and $\ol{\la_0}$, both of order $m+1$. Therefore, there exists a polynomial $q$ (with real coefficients) with $q(\la_0)\neq 0$ (and hence also $q(\ol{\la_0})\neq 0$) such that
$$
g(\la) = q(\la)(\la - \la_0)^{-m-1}(\la - \ol{\la_0})^{-m-1}.
$$
From the identity
$$
\la^2p(\la) - r_1(\la_1)(\la - \la_0)^{m+1}(\la - \ol{\la_0})^{m+1} = q(\la)(\la - \la_1)
$$
we see that $\deg(q) = 2m+1$. Hence, the operator $g(AG)$ is bounded. Let $(x_n)\subset\dom AG$ be a sequence with $\|x_n\| = 1$ and $(AG - \la_1)x_n\to 0$ as $n\to\infty$. With the operator $A_0$ from \eqref{e:A0} we have
$$
(A_0G_0 - r_1(\la_1))x_n = (r_1(AG) - r_1(\la_1))x_n = g(AG)(AG - \la_1)x_n\to 0
$$
as $n\to\infty$. And since $G_0A_0G_0\ge 0$ it follows from $r_1(\la_1) > 0$ and Proposition \ref{p:bounded_non-negative} that
$$
\liminf_{n\to\infty}\,[x_n,x_n] > 0.
$$
This shows that $\la_1\in\sp(AG)$. The assertion (c) is proved similarly.
\end{proof}

The following example shows that the condition ($*$) is essential for Theorem \ref{t:main} to be valid.

\begin{ex}
Let $T$ be a closed and densely defined symmetric operator in the Hilbert space $\calH$ which is uniformly positive but not selfadjoint. Then $T$ has a uniformly positive selfadjoint extension $A$ (e.g., the Friedrichs extention). Since for $x\in\dom(T^*T)$ we have $(T^*Tx,x) = \|Tx\|^2\ge\delta\|x\|^2$ with some $\delta > 0$, the selfadjoint operator $|T| := (T^*T)^{1/2}$ is boundedly invertible. We set $G := |T|^{-1}$. Then $AG = T|T|^{-1}$ and hence $(AGx,Gx)\ge 0$ for $x\in\dom AG$. But since $AG$ is bounded while $A$ is unbounded, it follows from Remark \ref{r:no_star} that ($*$) is not satisfied. Let us now see that the statements (a)--(c) of Theorem \ref{t:main} do not apply. For this we note that for $x\in\dom |T|$ and $y\in\calH$ we have
$$
\big(T|T|^{-1}x,T|T|^{-1}y\big) = \big((T^*T)^{1/2}x,(T^*T)^{-1/2}y\big) = (x,y)
$$
which shows that the operator $AG$ is an isometry with $\dom AG = \calH$ and $\ran AG = \ran T\neq\calH$. The spectrum of $AG$ therefore coincides with the closed unit disk.
\end{ex}

Assume that $(A,G)$ is definitizable. Theorem \ref{t:main} shows that there is only a finite number of real points which are not contained in $\rho(AG)\cup\sp(AG)\cup\sm(AG)$. In analogy to definitizable operators in Krein spaces these exceptional points will be called the {\it critical points} of $(A,G)$. By Theorem \ref{t:main} each non-zero critical point of $(A,G)$ is a zero of every definitizing polynomial for $(A,G)$. Moreover, if $G$ is not boundedly invertible, then due to Propsition \ref{p:zero_res} and Corollary \ref{c:zero_+} zero is a critical point of $(A,G)$. The set of the critical points of $(A,G)$ is denoted by $c(A,G)$.

\begin{thm}\label{t:sf}
Assume that $(A,G)$ is definitizable. Then the operator $AG$ possesses a spectral function on $\R$ with the set of critical points $s := c(A,G)$.
\end{thm}
\begin{proof}
The proof is divided into several steps. In step 1 we define the spec\-tral projection $E(\Delta)$ for sets $\Delta$ which have a positive distance to $s$. In step 2, $E(\Delta)$ is defined for compact intervals. This will be used in step 3 to define $E(\Delta)$ for all $\Delta\in\mathfrak R_s(\R)$.

{\bf 1.} By $\mathfrak R_{s,0}(\R)$ we denote the system of all sets $\Delta$ in $\mathfrak R_s(\R)$ with $\Delta\cap s = \emptyset$. In this first step of the proof we define $E(\Delta)$ for $\Delta\in\mathfrak R_{s,0}(\R)$ and prove that the set function $E$ on $\mathfrak R_{s,0}(\R)$ satisfies (S1)--(S5) in Definition \ref{d:sf}. Let $p$ be a definitizing polynomial for $(A,G)$ and let $Z$ be the set of zeros of $p$. By Theorem \ref{t:main} the points in $Z$ divide the real line into intervals which are of either positive or negative type with respect to $AG$. The set $Z$ contains the critical points of $(A,G)$, but there might be spectral points of $AG$ in $Z$ which are not critical. However, a slight modification of the set $Z$ leads to a finite set $Z'$ of real points which divide $\R$ into intervals $J_1,\ldots,J_n$ of positive or negative type with respect to $AG$, respectively, such that $Z'\cap\sigma(AG) = s$. By Theorem \ref{t:lsf}, on each interval $J_k$ the operator $AG$ has a local spectral function $E_k$. For $\Delta\in\mathfrak R_{s,0}(\R)$ we set $\Delta_k := \Delta\cap J_k\cap\sigma(AG)$, $k=1,\ldots,n$, and
$$
E(\Delta) := \sum_{k=1}^n\,E_k(\Delta_k).
$$
As $\Delta_k\in\mathfrak R(J_k)$ for $k=1,\ldots,n$, this is a proper definition. Each of the subspaces $\calL_k := E_k(\Delta_k)$, $k=0,\ldots,n$, is contained in $\dom AG$ and is $AG$-invariant. In the following we shall show that $\calL_k\cap\calL_j = \{0\}$ for $k\neq j$. Let $\la\in\C$ be arbitrary. Then $\la\notin\ol{\Delta_k}$ or $\la\notin\ol{\Delta_j}$. Assume $\la\notin\ol{\Delta_j}$. Then $\la\in\rho(AG|\calL_j)$ and thus $\ker(AG|\calL_k\cap\calL_j - \la) = \{0\}$. Let $y\in\calL_k\cap\calL_j$. Then, as $y\in\calL_j$, the vector
$$
x := (AG|\calL_j - \la)^{-1}y = \lim_{\eta\downto 0}\,(AG - (\la + i\eta))^{-1}y
$$
exists and is contained in both $\calL_j$ and $\calL_k$. Hence, we have $\la\in\rho(AG|\calL_k\cap\calL_j)$. As this is similarly proved for $\la\notin\ol{\Delta_k}$, it follows that $\sigma(AG|\calL_k\cap\calL_j) = \emptyset$ and hence $\calL_k\cap\calL_j = \{0\}$. Therefore, as $E_k(\Delta_k)$ and $E_j(\Delta_j)$ commute, we obtain
$$
E_k(\Delta_k)E_j(\Delta_j) = E_j(\Delta_j)E_k(\Delta_k) = 0.
$$
This shows that $\calL_k\ds\calL_j$ is a subspace and that $\calL_k\subset\calL_j^\gperp$. In fact, we have shown that
$$
E(\Delta)\calH = E_0(\Delta_0)\calH\,[\ds]\,\ldots\,[\ds]\,E_n(\Delta_n)\calH.
$$
With the help of this decomposition it is easily seen that the function $E$, defined on $\mathfrak R_{s,0}(\R)$, satisfies (S1)--(S5) in Definition \ref{d:sf}.

{\bf 2.} In this step we define the spectral projection $E([a,b])$ for a compact interval $[a,b]\in\mathfrak R_s(\R)$. To this end choose $a',b'$ with $a < a' < b' < b$ such that there is no critical point of $AG$ in $[a,a']\cup [b',b]$. We set
$$
\Delta_0 := [a,a']\quad\text{and}\quad\Delta_1 := [b',b].
$$
Define the spectral subspaces $\calL_j := E(\Delta_j)\calH$, $j=0,1$. As these are both uniformly definite, on account of Lemma \ref{l:ks->oc} we have
\begin{equation}\label{e:decomp}
\calH = \calL_0\,[\ds]\,\calL_1\,[\ds]\,\wt\calH,
\end{equation}
where $\wt\calH = (\calL_0[\ds]\calL_1)^\gperp = (I - E(\Delta_0\cup\Delta_1))\calH$. We set $T_j := AG|\calL_j$, $j=0,1$, and $\wt T := AG|\wt\calH$. With respect to the decomposition \eqref{e:decomp} the operator $AG$ decomposes as $AG = T_0\,[\ds]\,T_1\,[\ds]\,\wt T$. As a consequence of the results in step 1 we have
\begin{equation}\label{e:later}
\sigma(\wt T)\,\subset\,\ol{\sigma(AG)\setminus(\Delta_0\cup\Delta_1)}.
\end{equation}
This implies $(a,a')\cup (b',b)\subset\rho(\wt T)$. Set $\Delta := [a',b']$ and denote by $\wt E_\Delta$ the Riesz-Dunford spectral projection of $\wt T$ (in $\wt\calH$) corresponding to $\Delta$. Similarly as in the proof of Lemma \ref{l:E} it is seen that $\wt E_\Delta$ is $\product$-symmetric. With respect to the decomposition \eqref{e:decomp} we now define
$$
E([a,b]) := I_{\calL_0}\,[\ds]\,I_{\calL_1}\,[\ds]\,\wt E_\Delta.
$$
This is obviously a $\product$-symmetric projection in $\calH$ which commutes with the resolvent of $AG$. Moreover, $\sigma(AG|E([a,b])\calH)\subset[a,b]$.

In the following we show that the above definition of $E([a,b])$ is independent of the choice of $a'$ and $b'$. To this end we prove the following claims.
\begin{enumerate}
\item[(C1)] The subspace $E([a,b])\calH$ is the maximal spectral subspace of $AG$ corresponding to $[a,b]$.
\item[(C2)] $E([a,b])$ commutes with every bounded operator which commutes with the resolvent of $AG$.
\end{enumerate}
For the proof of (C1) let $\calK\subset\dom AG$ be an $AG$-invariant (closed) subspace such that $\sigma(AG|\calK)\subset [a,b]$. By Theorem \ref{t:lsf} the maximal spectral subspaces $\calK_j$ of $AG|\calK$ corresponding to $\Delta_j$ exist, $j=0,1$. These are uniformly definite with respect to the inner product $\product$. Hence,
$$
\calK = \calK_0\,[\ds]\,\calK_1\,[\ds]\,\wt\calK,
$$
where $\wt\calK = (\calK_0\,[\ds]\,\calK_1)^\gperp\cap\calK$ and $\sigma(AG|\wt\calK)\subset [a',b']$. From $\sigma(AG|\calK_j)\subset\Delta_j$ and the maximality of $\calL_j$ we conclude $\calK_j\subset\calL_j$, $j=0,1$, and set
$$
\calM := (\calL_0\,[\ds]\,\calL_1) + \wt\calK.
$$
This sum is direct (and hence $\sigma(AG|\calM)\subset [a,b]$): Set $\calL := \calL_0[\ds]\calL_1$. By \cite[Theorem 0.8]{rr}, $\sigma(AG|\calL\cap\wt\calK)\subset (\Delta_0\cup\Delta_1)\cap [a',b'] = \{a',b'\}$. From the maximality of $\calK_0$ and $\calK_1$ it follows that $a',b'\notin\sigma_p(AG|\calL\cap\wt\calK)$. And as the resolvent of $AG|\calL\cap\wt\calK$ satisfies a growth condition \eqref{e:growth} in neighborhoods of $\Delta_0$ and $\Delta_1$, we conclude $\calL\cap\wt\calK = \{0\}$.

Now, with $\wt\calM := (\calL_0[\ds]\calL_1)^\gperp\cap\calM$ we have
$$
\calM = \calL_0\,[\ds]\,\calL_1\,[\ds]\,\wt\calM.
$$
As $\calL_0$ and $\calL_1$ are maximal, the spectrum of $AG|\wt\calM$ is contained in $[a',b']$. Since $\wt\calM\subset\wt\calH$ and $\wt E_\Delta\wt\calH$ (as a Riesz-Dunford spectral subspace) is the maximal spectral subspace of $AG|\wt\calH$ corresponding to $[a',b']$, this implies $\wt\calM\subset\wt E_\Delta\wt\calH$ and hence $\calK\subset\calM\subset E([a,b])\calH$. (C1) is proved.

Let $B$ be a bounded operator in $\calH$ which commutes with the resolvent of $AG$. Then $BAG\subset AGB$ and hence $E(\Delta_j)B = BE(\Delta_j)$, $j=0,1$, see (S3) in Definition \ref{d:sf}. Hence, $\calL_0$ and $\calL_1$ and also their orthogonal companions $\calL_0^\gperp$ and $\calL_1^\gperp$ are $B$-invariant. And as $\wt\calH = \calL_0^\gperp\cap\calL_1^\gperp$, it follows that with respect to the decomposition \eqref{e:decomp} the operator $B$ decomposes as $B = B_0[\ds]B_1[\ds]\wt B$. Hence, $\wt B\wt T\subset\wt T\wt B$ which implies that $\wt B$ commutes with $\wt E_\Delta$. Finally, we conclude that $B$ commutes with $E([a,b])$, and (C2) is proved.

Now, let $a'',b''\in\R$ with $a < a'' < b'' < b$ such that $[a,a'']$ and $[b'',b]$ do not contain any point from $s$ and construct a spectral projection of $AG$ corresponding to $[a,b]$ as in step 1 with $a'$ and $b'$ replaced by $a''$ and $b''$. Denote this projection by $P$. As the maximal spectral subspace of $AG$ corresponding to $[a,b]$ is unique, we have $P\calH = E([a,b])\calH$ by (C1). Therefore, $PE([a,b]) = E([a,b])$ and $E([a,b])P = P$. But (C2) yields that $P$ and $E([a,b])$ commute. Therefore, $P = E([a,b])P = PE([a,b]) = E([a,b])$.

Above, it was shown that $E([a,b])$ commutes with any bounded operator in $\calH$ which commutes with the resolvent of $AG$ and that $\sigma(AG|E([a,b])\calH)\subset\sigma(AG)\cap [a,b]$ holds. Hence, the projection $E([a,b])$ has the properties (S3) and (S4) in Definition \ref{d:sf}. It also satisfies (S5) as due to $(a,a')\cup (b',b)\subset\rho(\wt T)$ and \eqref{e:later} we have
\begin{align*}
\sigma(AG|(I - E([a,b]))\calH)
&= \sigma(\wt T|(I - \wt E_\Delta)\calH) = \sigma(\wt T)\setminus (a,b)\\
&\subset \ol{\sigma(AG)\setminus (\Delta_0\cup\Delta_1)}\setminus (a,b) = \ol{\sigma(AG)\setminus [a,b]}.
\end{align*}
Moreover, similarly as the proof of $E_k(\Delta_k)E_j(\Delta_j) = 0$ in step 1, it is proved that $E([a,b])E([c,d]) = 0$ for compact intervals $[a,b],[c,d]\in\mathfrak R_s(\R)$ with $[a,b]\cap [c,d] = \emptyset$.

{\bf 3.} In this last step of the proof we define the spectral projection $E(\Delta)$ for every $\Delta\in\mathfrak R_s(\R)$ and show that the function $E$, defined on $\mathfrak R_s(\R)$, has the properties (S1)--(S5) in Definition \ref{d:sf}. Let $\Delta\in\mathfrak R_s(\R)$. Then each $\alpha\in\Delta\cap s$ is contained in the interior $\Delta^i$ of $\Delta$. Hence, there exists a compact interval $\Delta_\alpha\subset\Delta$ such that $\Delta_\alpha^i\cap s = \{\alpha\}$. Choose these intervals such that $\Delta_\alpha\cap\Delta_\beta = \emptyset$ for $\alpha,\beta\in\Delta\cap s$, $\alpha\neq\beta$, and define the projection $E(\Delta)$ by
\begin{equation}\label{e:final}
E(\Delta) := \sum_{\alpha\in\Delta\cap s}\,E(\Delta_\alpha) + E\left(\Delta\setminus\bigcup_{\alpha\in\Delta\cap s}\,\Delta_\alpha\right).
\end{equation}
Let $\alpha\in s$ and let $[a,b]\in\mathfrak R_s(\R)$ such that $(a,b)\cap s = \{\alpha\}$. Furthermore, let $a',b'\in (a,b)$ such that $a' < \alpha < b'$. From the construction of $E([a,b])$, $E([a',b])$ and $E([a,b'])$ in step 2 it is seen that
$$
E([a,a')) + E([a',b]) = E([a,b']) + E((b',b]) = E([a,b]).
$$
With the help of this property it is shown that $E(\Delta)$ in \eqref{e:final} is well-defined.

It remains to verify that $E$ satisfies the conditions (S1)--(S5) in Definition \ref{d:sf}. Let $\Delta_1,\Delta_2\in\mathfrak R_s(\R)$. Then $\Delta_j = \Delta_j^1\cup\Delta_j^2$, where $\Delta_j^1\cap\Delta_j^2 = \emptyset$, $\Delta_j^2\in\mathfrak R_{s,0}(\R)$ and
$$
\Delta_j^1 = \bigcup_{\alpha\in\Delta_j\cap s}\,\Delta_\alpha^j
$$
with compact intervals $\Delta_\alpha^j$ as above, $j=1,2$. We may choose the intervals $\Delta_\alpha^j$ such that the following holds:
\begin{enumerate}
\item[(a)] $\Delta_1^2\cap\Delta_2^1 = \Delta_1^1\cap\Delta_2^2 = \emptyset$,
\item[(b)] $\Delta_\alpha^1 = \Delta_\alpha^2$ for $\alpha\in\Delta_1\cap\Delta_2\cap s$,
\item[(b)] $\Delta_\alpha^1\cap\Delta_\beta^2 = \emptyset$ if $\alpha\neq\beta$.
\end{enumerate}
Then we have
\begin{align*}
E(\Delta_1\cap\Delta_2)
&= E((\Delta_1^1\cup\Delta_1^2)\cap(\Delta_2^1\cup\Delta_2^2))\\
&= E((\Delta_1^1\cap\Delta_2^1)\cup(\Delta_1^2\cap\Delta_2^2))\\
&= \sum_{\alpha\in\Delta_1\cap\Delta_2\cap s}\,E(\Delta_\alpha^1) + E(\Delta_1^2)E(\Delta_2^2).
\end{align*}
On the other hand,
$$
E(\Delta_1)E(\Delta_2) = \sum_{\alpha\in\Delta_1\cap s}\,\sum_{\beta\in\Delta_2\cap s}\,E(\Delta_\alpha^1)E(\Delta_\beta^2) + E(\Delta_1^2)E(\Delta_2^2).
$$
And as $E(\Delta_\alpha^1)E(\Delta_\beta^2) = \delta_{\alpha\beta}E(\Delta_\alpha^1)$, where $\delta_{\alpha\beta}$ is the Kronecker delta, (S1) follows.

The proof of (S2) is straightforward and (S3) follows from the facts proved in steps 1 and 2. For the proofs of (S4) and (S5) let $\Delta\in\mathfrak R_s(\R)$. Then $\Delta = \Delta_1\cup\Delta_2$ where $\Delta_1\cap\Delta_2 = \emptyset$, $\Delta_2\in\mathfrak R_{s,0}(\R)$, and $\Delta_1$ is the union of mutually disjoint compact intervals $\Delta_{\alpha_j}\in\mathfrak R_s(\R)$, $j=1,\ldots,r$, with $\Delta_{\alpha_j}\cap s = \{\alpha_j\}$. Due to the definition of $E(\Delta)$ we have
$$
E(\Delta)\calH = E(\Delta_{\alpha_1})\calH\,[\ds]\,\ldots\,[\ds]\,E(\Delta_{\alpha_r})\calH\,[\ds]\,E(\Delta_2)\calH.
$$
Hence,
\begin{align*}
\sigma(AG|E(\Delta)\calH)
&\subset(\sigma(AG)\cap\Delta_{\alpha_1})\cup\ldots\cup(\sigma(AG)\cap\Delta_{\alpha_1})\cup\ol{\sigma(AG)\cap\Delta_2}\\
&= (\sigma(AG)\cap\Delta_1)\cup\ol{\sigma(AG)\cap\Delta_2}\,\subset\,\ol{\sigma(AG)\cap\Delta}.
\end{align*}
From $(I - E(\Delta))\calH\subset (I - E(\Delta_{\alpha_j}))\calH$ for $j=1,\ldots,r$ and $(I - E(\Delta))\calH\subset (I - E(\Delta_2))\calH$ we conclude
\begin{align*}
\sigma(AG|(I - E(\Delta))\calH)&\,\subset\,\sigma(AG|(I - E(\Delta_{\alpha_j}))\calH)\\
\sigma(AG|(I - E(\Delta))\calH)&\,\subset\,\sigma(AG|(I - E(\Delta_2))\calH),
\end{align*}
and therefore
\begin{align*}
\sigma(AG|(I - E(\Delta))\calH)
&\subset \ol{\sigma(AG)\setminus\Delta_{\alpha_1}}\cap\ldots\cap\ol{\sigma(AG)\setminus\Delta_{\alpha_r}}\cap\ol{\sigma(AG)\setminus\Delta_2}\\
&\subset\ol{\sigma(AG)\setminus\Delta_1}\,\cap\,\ol{\sigma(AG)\setminus\Delta_2}\\
&\subset\ol{\sigma(AG)\setminus\Delta}\,\cup\,\partial\Delta_1,
\end{align*}
where $\partial\Delta_1$ is the real boundary of $\Delta_1$. This is a finite set which depends on the choice of the $\Delta_{\alpha_j}$'s. Hence, the theorem is proved.
\end{proof}

\section{An application to Sturm-Liouville problems}\label{s:application}
Let $w$, $p$ and $q$ be real-valued functions on a bounded or unbounded open interval $(a,b)$ such that $w,p^{-1},q\in L^1_{\rm loc}(a,b)$ and $w > 0$ almost everywhere. The differential expression
$$
\tau(f) := \frac 1 w\,\big(-(pf')' + qf\big)
$$
is then called a {\it Sturm-Liouville} differential expression. Usually, the differential operators associated with $\tau$ are considered in the weighted $L^2$-space $L^2_w(a,b)$ which consists of all (equivalence classes of) measurable functions $f : (a,b)\to\C$ for which $f^2w\in L^1(a,b)$. If
$$
\underset{x\in (a,b)}{{\rm ess}\inf}\;w(x) > 0\quad\text{and}\quad\underset{x\in (a,b)}{{\rm ess}\sup}\;w(x) < \infty,
$$
then the topologies of $L^2_w(a,b)$ and $L^2(a,b)$ coincide, and the selfadjoint realizations of $\tau$ in $L^2_w(a,b)$ are similar to selfadjoint operators in $L^2(a,b)$. In the following we use the abstract results from the previous section to show that also in more general cases it can make sense to consider differential operators associated with $\tau$ in (the unweighted space) $L^2(a,b)$.

By $A$ denote the operator of multiplication with the function $w^{-1}$ in the Hilbert space $L^2(a,b)$. The operator $A$ is selfadjoint and non-negative (in $L^2(a,b)$). In addition, define the operator $G_{\max}$ in $L^2(a,b)$ by $G_{\max}f := -(pf')' + qf$, $f\in\dom G_{\max}$, where
$$
\dom G_{\max} := \{f\in L^2(a,b) : f,pf'\in AC_{\rm loc}(a,b),\,-(pf')' + qf\in L^2(a,b)\}.
$$
The selfadjoint realizations of the differential expression
$$
\tau_0(f) := -(pf')' + qf
$$
in $L^2(a,b)$ are well-known to be restrictions of $G_{\max}$. In what follows let $G$ be a selfadjoint realization of $\tau_0$ in $L^2(a,b)$.

\begin{prop}\label{p:SL1}
If $w\in L^\infty(a,b)$ and $G$ is boundedly invertible, then the spectrum of the operator $AG$ is real, and $AG$ has a spectral function without singularities on $\R$.
\end{prop}
\begin{proof}
From $w\in L^\infty(a,b)$ it follows that the operator $A = w^{-1}$ is boundedly invertible in $L^2(a,b)$. Hence, $0\in\rho(A)\cap\rho(G)$ which implies that both $AG$ and $GA$ are boundedly invertible. Therefore, \eqref{e:ass} is satisfied for the selfadjoint operators $A$ and $G$. Furthermore, for $f\in\dom AG$ we have $(AGf,Gf)\ge 0$ as $A$ is non-negative. Hence, the pair $(A,G)$ is definitizable with definitizing polynomial $p(\la) = \la$, and the assertions follow directly from Theorems \ref{t:main} and \ref{t:sf}.
\end{proof}

\section*{Acknowledgements}
Friedrich Philipp gratefully acknowledges the support from Deutsche For\-schungs\-ge\-mein\-schaft (DFG) under grant BE 3765/5-1.

\section*{Contact information}
Tomas Ya.\ Azizov: Voronezh State University, Department of Mathematics, Universitetskaya pl.\ 1, 394006 Voronezh, Russia, {\tt azizov@math.vsu.ru}

\vspace{0cm}
Mikhail Denisov: Voronezh State University, Department of Mathematics, Universitetskaya pl.\ 1, 394006 Voronezh, Russia, {\tt denisov.m.1981@gmail.com}

\vspace{0cm}
Friedrich Philipp: Institut f\"ur Mathematik, MA 6-4, Technische Universit\"at Berlin, Stra\ss e des 17.\ Juni 136, 10623 Berlin, Germany, {\tt fmphilipp@gmail.com}

\begin{thebibliography}{99}
\bibitem{ai}   T.Ya.\ Azizov and I.S.\ Iokhvidov,
               Linear Operators in Spaces with an Indefinite Metric.
               John Wiley \& Sons, Ltd., Chichester, 1989.

\bibitem{abjt} T.Ya.\ Azizov, J.\ Behrndt, P.\ Jonas and C.\ Trunk,
               Spectral points of type $\pi_+$ and type $\pi_-$ for closed linear relations in Krein spaces.
               J.\ London Math.\ Soc.\ {\bf 83} (2011), 768-788.

\bibitem{abt}  T.Ya.\ Azizov, J.\ Behrndt and C.\ Trunk,
               On finite rank perturbations of definitizable operators.
               J.\ Math.\ Anal.\ Appl.\ {\bf 339} (2) (2008), 1161--1168.

\bibitem{ajt}  T.Ya.\ Azizov, P.\ Jonas and C.\ Trunk,
               Spectral points of type $\pi_{+}$ and $\pi_{-}$ of self-adjoint operators in Krein spaces.
               J.\ Funct.\ Anal.\ {\bf 226} (2005), 114--137.

\bibitem{bp}   J.\ Behrndt and F.\ Philipp,
               Spectral analysis of singular ordinary differential operators with indefinite weights.
               J.\ Differential Equations {\bf 248} (2010), 2015--2037.

\bibitem{b}    J.\ Bogn\'ar,
               Indefinite Inner Product Spaces.
               Springer-Verlag, New York-Heidelberg, 1974.

\bibitem{cf}   I.\ Colojoar\u{a} and C.\ Foia\c{s},
               Theory of Generalized Spectral Operators.
               Gordon and Breach, Science Publishers, Inc., New York (1968).

\bibitem{c}    J.B.\ Conway,
               A Course in Functional Analysis,
               Springer, New York (1985).

\bibitem{cl}   B.\ \'Curgus and H.\ Langer,
               A Krein space approach to symmetric ordinary differential operators with an indefinite weight function.
               J.\ Differential Equations {\bf 79} (1989), 31--61.

\bibitem{denm} M.\ Denisov,
               The spectral function for some products of self-adjoint operators.
               Matematicheskie Zametki {\bf 81} (6) (2007), 948--951.

\bibitem{dl}   A.\ Dijksma and H.\ Langer,
               Operator theory and ordinary differential operators.
               In: B\"{o}ttcher, Albrecht (ed.) et al., Lectures on operator theory and its applications.
               Fields Institute for Research in Mathematical Sciences, Toronto, Canada 1994.
               Providence, RI: American Mathematical Society.
               Fields Inst.\ Monogr.\ {\bf 3} (1996), 75--139.

\bibitem{dsn}  A.\ Dijksma and H.S.V.\ de Snoo,
               Symmetric and selfadjoint relations in Krein spaces I.
               Oper.\ Theory Adv.\ Appl.\ {\bf 24} (1987), 145--166.

\bibitem{ds}   N.\ Dunford and J.T.\ Schwartz,
               Linear Operators Part I: General Theory.
               Wiley, New York (1958).

\bibitem{hkm}  V.\ Hardt, A.\ Konstantinov and R.\ Mennicken,
               On the spectrum of the product of closed operators.
               Math.\ Nachr.\ {\bf 215} (2000), 91--102.

\bibitem{hm}   V.\ Hardt and R.\ Mennicken,
               On the spectrum of unbounded off-diagonal $2\times 2$ operator matrices in Banach spaces.
               Oper.\ Theory Adv.\ Appl.\ {\bf 124} (2001), 243--266.

\bibitem{j1}   P.\ Jonas,
               On the functional calculus and the spectral function for definitizable operators in Krein space.
               Beitr\"age zur Analysis {\bf 16} (1981), 121--135.

\bibitem{j2}   P.\ Jonas,
               On locally definite operators in Krein spaces.
               In: Spectral Theory and Applications,
               Theta Ser.\ Adv.\ Math.\ {\bf 2} (2003),
               Theta, Bucharest (95--127).

\bibitem{jl}   P.\ Jonas and H.\ Langer,
               Compact perturbations of definitizable operators.
               J.\ Operator Theory {\bf 2} (1979), 63--77.

\bibitem{jt02} P.\ Jonas and C.\ Trunk,
               On a class of analytic operator functions and their linearizations.
               Math.\ Nachr.\ {\bf 243} (2002), 92--133.

\bibitem{jt06} P.\ Jonas and C.\ Trunk,
               A Sturm-Liouville problem depending rationally on the eigenvalue parameter.
               Math.\ Nachr.\ {\bf 280} (15) (2007), 1709-�1726.

\bibitem{km}   I.M.\ Karabash and M.M.\ Malamud,
               Indefinite Sturm-Liouville operators $({\rm sgn}\,x)(-\tfrac{d^2}{dx^2}+q(x))$ with finite-zone potentials.
               Oper.\ Matrices {\bf 1} (2007), 301--368.

\bibitem{kt}   I.M.\ Karabash and C. Trunk,
               Spectral properties of singular Sturm-Liouville operators.
               Proc.\ Roy.\ Soc.\ Edinburgh Sect.\ A {\bf 139} (2009), 483--503.

\bibitem{k}    T.\ Kato,
               Perturbation Theory for Linear Operators.
               Second Edition, Springer, 1976.

\bibitem{kwz}  Q.\ Kong, H.\ Wu and A.\ Zettl,
               Singular left-definite Sturm-Liouville problems.
               J.\ Differential Equations {\bf 206} (2004), 1--29.

\bibitem{lamm} P.\ Lancaster, A.S.\ Markus and V.I.\ Matsaev,
               Definitizable operators and quasihyperbolic operator polynomials.
               J.\ Funct.\ Anal.\ \textbf{131} (1995), 1--28.

\bibitem{l_h}  H.\ Langer,
               Spektraltheorie linearer Operatoren in $J$-R\"{a}umen und einige Anwendungen auf die Schar $L(\lambda) = \lambda^2I + \lambda B + C$.
               Habilitationsschrift, Technische Universit\"{a}t Dresden (1965).

\bibitem{l}    H.\ Langer,
               Spectral functions of definitizable operators in Krein spaces.
               Functional analysis (Dubrovnik, 1981),
               Lect.\ Notes Math.\ {\bf 948} (1982), 1--46.

\bibitem{lmm}  H.\ Langer, A.S.\ Markus and V.I.\ Matsaev,
               Locally definite operators in indefinite inner product spaces.
               Math.\ Ann., {\bf 308} (1997), 405--424.

\bibitem{lmem} H.\ Langer, R.\ Mennicken and M.\ M\"{o}ller,
               A second order differential operator depending nonlinearly on the eigenvalue parameter.
               Oper.\ Theory Adv.\ Appl.\ {\bf 48} (1990), 319--332.

\bibitem{lm}   I.\ Lyubich and V.I.\ Matsaev,
               On operators with decomposable spectrum.
               Mat.\ Sbornik \textbf{56} (98) (1962), 433--468 (Russian).
               Engl.\ transl.: AMS Transl.\ (2) \textbf{47} (1965), 89--129.
%
%
%
\bibitem{rr}   H.\ Radjavi and P.\ Rosenthal,
               Invariant Subspaces.
               Dover Publications, Inc., Mineola, New York, 2003.
\end{thebibliography}
\end{document}